\documentclass[11pt]{amsart}
\usepackage{amsfonts}
\usepackage{amsmath}
\usepackage{amssymb}
\usepackage{amsthm}
\usepackage{ color, ulem}

\input xy
\xyoption {all}

\def \AA {{\mathsf{A}}^\star}
\def \RR {{\mathsf{R}}^\star}
\def \NL {{\mathsf{NL}}^\star}

\def \prankpi {{\mathcal Q^{\,\pi}_{\,dH, \,\chi}(\mathbb C^r)}}
\def \pranktwopi {{\mathcal Q^{\,\pi}_{\,H, \,\chi}(\mathbb C^2)}}
\def \pranktwo {{\mathcal Q_{H, \,\chi}(\mathbb C^2)}}
\def \pranktwor {{\mathcal Q_{H, \,\chi}(\mathbb C^r)}}
\newcommand{\comment}[1]{}
\newtheorem{theorem}{Theorem}
\newtheorem {lemma}{Lemma}
\newtheorem{conjecture}{Conjecture}
\newtheorem {corollary}{Corollary}
\newtheorem {proposition}{Proposition}
\theoremstyle{definition}
 
\theoremstyle {definition}
\newtheorem{remark}{Remark}

\begin{document}
\baselineskip=16pt

\title[Segre classes and Hilbert schemes of points]{Segre classes and   
Hilbert schemes of points}

\author{A. Marian}
\address{Department of Mathematics, Northeastern University}
\email {a.marian@neu.edu}
\author{D. Oprea}
\address{Department of Mathematics, University of California, San Diego}
\email {doprea@math.ucsd.edu}
\author{R. Pandharipande}
\address{Department of Mathematics, ETH Z\"urich}
\email {rahul@math.ethz.ch}

\begin{abstract} We prove a closed
formula for the integrals of the top Segre classes of 
tautological bundles over the Hilbert schemes of points of a 
$K3$ surface $X$.
 We derive relations among 
the Segre classes 
via equivariant localization of the virtual fundamental classes of 
Quot schemes on $X$. 
The resulting recursions are then
solved explicitly. The formula proves the
$K$-trivial case of a conjecture of M. Lehn from 1999. 

The relations determining the Segre classes fit into a much wider 
theory. By localizing the virtual classes of certain relative Quot schemes on surfaces, we obtain new systems of relations among tautological 
classes on
moduli spaces of surfaces and their relative Hilbert schemes of points.
For the moduli of $K3$ sufaces, we produce relations
intertwining the $\kappa$ classes and the Noether-Lefschetz loci.
Conjectures are proposed.
\end{abstract} 
\maketitle

\setcounter{section}{-1}
\section{Introduction} 
\label{iiii}
\subsection {Segre classes} 
Let $(S, H)$ be a pair consisting of a nonsingular
 projective surface $S$ and a line bundle
$$H \rightarrow S\, .$$ 
The {\it degree} of the pair $(S,H)$ is defined via the intersection
product on $S$,
$$H\cdot H= \int_S H^2  \in \mathbb{Z}\, .$$

 The Hilbert scheme of points $S^{[n]}$ carries a tautological rank $n$ vector bundle $H^{[n]}$ whose fiber over $\zeta \in S^{[n]}$ is given by 
$$\zeta\mapsto H^0(H\otimes {\mathcal O}_{\zeta})\, .$$ 
The top Segre class $$N_{S,H,n}=\int_{S^{[n]}} s_{2n} (H^{[n]})$$ 
appeared first in 
 the algebraic study 
of Donaldson invariants  
via the moduli space of rank $2$ bundles on $S$ \cite {Ty}. 
Such Segre classes play a basic role in the
Donaldson-Thomas counting of sheaves (often
entering via the obstruction theory).  
A classical interpretation of $N_{S,H,n}$ is also available.
If $|H|$ is a linear system of dimension $3n-2$ which induces
a  map $$S\rightarrow \mathbb P^{3n-2}\, ,$$ 
$N_{S,H,n}$ counts the $n$-chords of dimension $n-2$ to the image of $S$.

The main result of the paper is the calculation of 
the top Segre classes for all pairs $(X,H)$ in the $K3$ case.

\begin{theorem}\label{main}
If $(X, H)$ is a nonsingular $K3$ surface of degree $2\ell$, then $$
\int_{X^{[n]}} s_{2n} (H^{[n]})=2^{n} \binom{\ell-2n+2}{n}\, .$$ 
\end{theorem}

\subsection{Lehn's conjecture}
Let $S$ be a nonsingular projective surface.
The Segre class $N_{S,H,n}$ 
can be expressed as a polynomial of degree $n$ in  
the four variables $$H^2\, ,\ H\cdot K_S\,,\ K_S^2\,,\ c_2(S)\, ,$$ 
see \cite {T} for a proof. Furthermore, the 
form  \begin{equation}\label{egleq}\sum_{n=0}^{\infty} N_{S,H,n}\, z^n=
\exp\Big(H^2 \cdot A_1(z) +  (H\cdot K_S) \cdot A_2(z)+ K_S^2 \cdot A_3(z)+ c_2(S) \cdot A_4(z)\Big)\end{equation} in terms of
 four universal power series $A_1, A_2, A_3, A_4$ was proven
in \cite{EGL}. 
 The formulas for the four power series were explicitly conjectured by M. Lehn
in 1999.  

\begin{conjecture}[Lehn \cite {L}] We have \begin{equation}\label{series}\sum_{n=0}^{\infty} N_{S,H,n}\, z^n=\frac{(1-w)^{a}(1-2w)^{b}}{(1-6w+6w^2)^{c}}\, \end{equation} for the change of variable $$z=\frac{w(1-w)(1-2w)^4}{(1-6w+6w^2)^3}\, $$ and constants $$a=H\cdot K_S-2K_S^2\, ,\,\,\ b=(H-K_S)^2+3\chi(\mathcal O_S)\, ,\,
\, \
c=\frac{1}{2}H(H-K_S)+\chi(\mathcal O_S)\, .$$ 
\end{conjecture} 
As usual in the study of the Hilbert scheme of points, 
Theorem \ref{main} determines two of the power series in \eqref {egleq}. 
Specifically, Theorem \ref{main}
implies \begin{equation}\label{a1}A_1\left(\frac{1}{2} t(1+t)^2\right)=\frac{1}{2} \log (1+t)\, ,\end{equation} \begin{equation}\label{a4}A_4\left(\frac{1}{2} t(1+t)^2\right)=\frac{1}{8} \log (1+t)-\frac{1}{24} \log (1+3t)\, .\end{equation} 
The evaluation of $A_1$ and $A_4$ proves
 Lehn's conjecture for all surfaces with numerically
trivial canonical bundle.

\begin{corollary}\label{maincor} 
If $(A, H)$ is an abelian or bielliptic surface of degree $2\ell$, then $$\int_{A^{[n]}} s_{2n} (H^{[n]})=\frac{2^{n}\ell}{n} \binom{\ell-2n-1}{n-1}\, .$$ \end{corollary} 

\begin{corollary}\label{maincor2} 
If $(E, H)$ is an Enriques surface of degree $2\ell$, then 
$$\left(\sum_{n=0}^{\infty} {N_{E,H,n}}\, z^n\right)^2=\sum_{n=0}^\infty 2^n \binom{2\ell-2n+2}{n} z^n 
\, .$$ 
\end{corollary} 

\subsection {Strategy of the proof} The intersection 
theory of the Hilbert scheme of points can be approached
 via the inductive recursions set up in \cite {EGL} or via the Nakajima calculus \cite {L,N}. 
By these methods, the integration of tautological classes
is reduced to a combinatorial problem. 
Another strategy is to prove an equivariant
version of Lehn's conjecture 
for the Hilbert scheme of points of $\mathbb{C}^2$ via
appropriately weighted sums over partitions.
However, we do not know how to prove Theorem 1 along these lines.{\footnote{The
parallel problem in dimension 1, the calculation of the Segre
classes of tautological bundles over Hilbert schemes
of points of nonsingular {\it curves}, has been 
solved in \cite{Cott, LeB, Wang}.}}

Let $(X, H)$ be a nonsingular projective $K3$ surface.
We consider 
integrals over the Quot scheme $\pranktwo$ parametrizing quotients $$\mathbb C^2 \otimes \mathcal O_X \to F\to 0$$  where $F$ is a rank 0 coherent sheaf 
satisfying 
$$c_1(F)=H\ \ \  \text{and} \ \ \ \chi(F)=\chi\, .$$ The Quot scheme admits a reduced virtual class, and the integrals $$\int_{\left[\pranktwo\right]^{\text{red}}}\gamma\cdot
 0^{k}  $$  vanish for all $k> 0$ and for all choices of Chow classes $\gamma$. Here, the notation $0$ stands for the first Chern class of the trivial line bundle 
 $$c_1(\mathcal O)=0\in A^1(\pranktwo).$$ Virtual localization \cite{GP} with respect to a $\mathbb{C}^\star$
 action applied to the above integrals yields linear recursions between the expressions $$\int_{X^{[n]}} c_{n-i}(H^{[n]}) s_{n+i} (H^{[n]})\, . $$ 
The linear recursions are trivial for all but finitely many values of $H^2=2\ell$. The nontrivial recursions  can be solved to show that the top Segre integrals vanish for the values $$2n-2\leq \ell\leq 3n-3\, .$$ These
vanishings determine the intersections up to an ambiguity given by the leading term of a polynomial which we can calculate explicitly. 

\subsection {The moduli space of $K3$ surfaces} The relations used
to prove Theorem 1 fit into a wider 
program aimed at studying the tautological rings of the moduli space of surfaces. Consider the relative Quot scheme $\mathcal Q^{\text{rel}}_{H, \,\chi} (\mathbb C^2)$ over a family of smooth surfaces. We evaluate on the base, via equivariant localization,  the vanishing pushforwards 
$$\pi_*\left( \gamma\cdot 0^{k} \, \cap \,
{\left[\mathcal Q^{\text{rel}}_{H, \,\chi} (\mathbb C^2)\right]^{\text{vir}}}
\right),$$ for all $k> 0$ and various choices of Chow classes $\gamma$. The Segre classes of the tautological bundles over the relative Hilbert schemes of points appear naturally in the localization output. In cohomological degree zero on the base, the resulting equations lead to Lehn's formulas above. In higher cohomological degree, the analysis of the localization output is increasingly harder, and the calculations are more intricate. They give rise to new and rich systems of relations among tautological classes on moduli spaces of surfaces and their relative Hilbert schemes of points. 

We illustrate this program in Section \ref{MMM} by concrete examples for the moduli of quasipolarized $K3$ sufaces. We obtain in this fashion relations
intertwining the $\kappa$ classes and the Noether-Lefschetz loci. These calculations point to general conjectures.

\subsection {Plan of the paper}
Section \ref{111} concerns localization on Quot schemes.
Foundational aspects of the
 virtual classes of Quot schemes  on surfaces $S$
are discussed in Section \ref{alice}.
The 
virtual localization formula for the $\mathbb{C}^\star$
action on the Quot schemes of $K3$ surfaces is presented
in Section \ref{bob}.
Explicit localization relations are derived in Section \ref{cat}.

Theorem \ref{main} is proven in Section \ref{pppp} by solving the
recursion relations of Section \ref{cat}. 
In Section \ref{LLL}, the connections between Theorem \ref{main}
and Lehn's conjecture are explained (and Corollaries 1 and 2 are
proven). An application to elliptically fibered surfaces 
is given in Corollary 3 of Section \ref{LLL}. 

In Section \ref{MMM}, we discuss the tautological classes of the moduli of $K3$ surfaces. In particular, we write down relations in the tautological ring and formulate conjectures.

\subsection {Acknowledgements} 
We thank N. Bergeron, G. Farkas, M. Lehn, D. Maulik, G. Oberdieck, and Q. Yin
for several discussions related to
tautological classes, Quot schemes, and the moduli space of 
$K3$ surfaces. The study of the
relations presented here was undertaken during 
a visit of A.M. and D.O. in the spring of 2015 to the {\it Institute for
Theoretical Sciences} at ETH Z\"urich (and supported
in part by SwissMAP).

A.M. was supported by the NSF through grant DMS 1303389. D. O. was supported by the Sloan Foundation and the NSF through grants DMS 1001486 and DMS 1150675. 
R.P. was supported by the Swiss National Science Foundation and
the European Research Council through
grants SNF-200021-143274 and ERC-2012-AdG-320368-MCSK.
R.P was also supported by SwissMAP and the Einstein Stiftung in Berlin.

\section {Localization on the Quot scheme}  \label{111}

\subsection{Geometric setup.} \label{alice}
Let $S$ be a nonsingular
 projective surface equipped with a divisor class $H$. We consider the 
Quot scheme $\pranktwor$ parametrizing short exact sequences $$0\to E\to \mathbb C^r\otimes \mathcal O_{S} \to F\to 0$$ where $F$ is a rank 0
coherent sheaf satisfying 
$$c_1 (F) = H\ \ \ \text{and} \ \ \ \chi (F) = \chi\, .$$ 

With the exception of the Hilbert scheme of points
$$S^{[n]} = {{\mathcal Q_{\,0, \,\chi}(\mathbb C^1)}}\, $$ 
the intersection theory of Quot schemes over surfaces has not
been extensively studied. 
Rank $1$ calculation can be found in \cite{DKO}, and
higher rank calculations over del Pezzo surfaces were 
considered in \cite{S}.

In comparison, the intersection theory of the Quot scheme of a {\it curve}  may be 
pursued in a virtual sense for a fixed curve \cite{MO} or  
by letting the curve vary via the moduli space of
stable quotients \cite {MOP}. The relations in the
tautological ring of the moduli of curves \cite{Ja,P,PP} 
via virtual localization  on the moduli 
stable quotients are parallel to the relations we introduce in Section \ref{MMM}.

Fundamental to our study is the following result (which we will
use here only in the  $r=2$ case).

\begin{lemma} 
The Quot scheme $\pranktwor$ admits a canonical
perfect obstruction theory 
of virtual dimension  $r\chi+ H^2$. 
\end{lemma} 

\proof Since the details are similar to the curve case \cite {MO}, we only discuss the main points. 
The obstruction theory of the Quot scheme is governed by the groups $\text{Ext}^{i}(E, F)$. 
We claim the vanishing $$\text{Ext}^2(E, F)=\text{Ext}^0(F, E\otimes K_S)^{\vee}=0\, .$$ Indeed, since $E$ is a subsheaf of $\mathbb C^r\otimes \mathcal O_S$, the latter group injects $$\text{Ext}^0(F, E\otimes K_S)\hookrightarrow \text{Ext}^0(F, \mathbb C^r\otimes K_S)=0\, ,$$ where the last vanishing follows since $F$ is torsion. As a consequence, the difference 
\begin{equation}\label{ddff}
\text{Ext}^0(E, F)-\text{Ext}^1(E, F)=\chi(E, F)=  r\chi+ H^2
\end{equation}
is constant. 

Since the higher obstructions vanish, the
 moduli space $\pranktwor$ carries a virtual fundamental class
of dimension \eqref{ddff}. 
\qed
\vskip.1in
Let $(X, H)$ be a primitively polarized $K3$ surface of degree $2\ell$
and
Picard rank 1. 
Let $$H^2 = 2\ell \, \ \text{and}\ n=\chi+\ell\, .$$ 
In the $K3$ case, we show the Quot scheme admits a reduced virtual fundamental class.

\begin{lemma} \label{redclass} For a $K3$ surface $(X, H)$ there is a natural surjective map $$\text{Ext}^{1}(E, F)\to \mathbb C\, ,$$ and a reduced virtual fundamental class $\left [\pranktwor\right]^{\text{red}}$ of dimension $r\chi+2\ell+1$.
\end{lemma} 

\proof The argument is standard. Indeed, the defining short exact sequence $$0\to E\to \mathbb C^r\otimes \mathcal O_X\to F \to 0$$  induces a natural morphism $$\text{Ext}^1(E, F)\to \text{Ext}^2(F, F)\stackrel{\text{Trace}}{\to} H^2(\mathcal O_X)=\mathbb C\, .$$ To prove surjectivity of the composition, 
it suffices to show $$\text{Ext}^1(E, F)\to \text{Ext}^2(F, F)$$ is surjective, since the trace is surjective. The cokernel of the map is identified with $$\text{Ext}^2(\mathbb C^r, F)=H^2(F)\otimes \mathbb C^r$$ which vanishes since $F$ has 1-dimensional support. The reduced virtual dimension equals 
$$\chi(E, F)+1=\chi(\mathbb C^r, F)-\chi(F, F)+1=r\chi+H^2+1
=r\chi+2\ell+1\, ,$$ 
where the summand $1$ on the left is due to the reduction.
\qed
\vskip.1in
In the $r=2$ case, the virtual dimension formula specializes to 
$$2(\chi+\ell)+1=2n+1\, .$$ 

\subsection{Virtual localization} \label{bob}
We continue to work with a 
primitively polarized $K3$ surface $(X,H)$ of degree $2\ell$
and Picard rank 1.

We study the virtual intersection theory of the Quot scheme $\pranktwo$ via 
the equivarint localization. To this end, consider the diagonal $\mathbb C^{\star}$ action, 
$$\mathbb C^2=\mathbb C[0]+\mathbb C[1]\, ,$$ with weights $0$ and $1$ on the two summands. 
A torus action on $\pranktwo$ is canonically induced. For a top degree Chow class $\alpha$ over the Quot scheme, the virtual localization formula of \cite {GP} reads: \begin{equation}\label{gpvi}\int_{\left[\pranktwo\right]^{\text{vir}}}\alpha=\sum_{\mathsf F} \int_{\left[\mathsf F\right]^{\text{vir}}}\frac{\tilde \alpha|_{\mathsf F}}{{\mathsf e}(\mathsf N_{\mathsf F}^{\text{vir}})}.\end{equation} On the right side, $\tilde \alpha$ is any lift of $\alpha$ to equivariant Chow. In addition, 
\begin {itemize}
\item the $\mathsf F$'s are the torus fixed loci of $\pranktwo$, 
\item $\mathsf N_{\mathsf F}^{\text{vir}}$ are the virtual normal complexes of the fixed loci,
\item ${\mathsf e}(\,)$ stands for the equivariant Euler class.
\end{itemize} For nonsingular projective varieties (endowed with the trivial perfect obstruction theory), equation \eqref{gpvi} specializes to the standard Atiyah-Bott localization formula \cite {AB}.

We turn now to the analysis of the fixed point loci. The torus fixed quotients come from kernels which split into rank $1$ factors $$0\to E=E_1\oplus E_2\hookrightarrow \mathcal O_X\oplus \mathcal O_X\to F=F_1\oplus F_2\to 0\, .$$ 
Since $c_1(F)=H$, we have  $$-c_1(E_1)-c_1(E_2)=H\, .$$ 
The two sections show that $E_1^{\vee}$ and $E_2^{\vee}$ are effective 
line bundles. 
Since the curve class $H$ is irreducible, the kernel must split as 
$$E_1=\mathcal I_Z, \ \  E_2=\mathcal I_W\otimes \mathcal O_X (-C)
\ \ \ \ \ \text{or} \ \ \ \ \ 
E_1=\mathcal I_Z\otimes \mathcal O_X(-C), \ \  
E_2=\mathcal I_W\, ,
 $$ where $C$ is a curve in the linear series $|H|$ and
 $$Z,W\subset X$$ are 0-dimensional subschemes of lengths $z$ and $w$ respectively. The condition $\chi(F)=\chi$ is equivalent to $$z+w=\ell+\chi=n\, .$$

For each value of $z$ and $w$, we obtain two distinct fixed loci, both isomorphic to the product $$X^{[z]}\times X^{[w]} \times \mathbb P$$ where the projective space $$\mathbb P=|H|$$ denotes the linear series in which the curve $C$ varies. 
\vskip.1in
We now turn to the virtual class and the virtual normal bundle of the fixed loci.  
Except in the two cases
\begin{enumerate}
\item[$\bullet$]
$(z, w)=(0, n)$ and $E_1\oplus E_2 = \mathcal O_X \ \oplus\ 
\mathcal I_W\otimes \mathcal O_X (-C)$ ,
\item[$\bullet$]
$(z, w)=(n, 0)$ and $E_1\oplus E_2 =  
\mathcal I_Z\otimes \mathcal O_X (-C)\ \oplus\ \mathcal O_X$ , 
\end{enumerate}
the reduced obstruction bundle has an {\it additional} trivial summand 
which forces the  localization contribution in \eqref{gpvi} to vanish.  
The non-reduced obstruction bundle is obtained by considering the fixed part of $\text{Ext}^1(E, F)$ which splits as 
$$\text{Ext}^1(E_1, F_1)+\text{Ext}^1(E_2, F_2)\, .$$ 
By the proof of Lemma \ref{redclass},
if $F_1$ and $F_2$ are both non-zero,
 we have two surjective maps $$\text{Ext}^1(E_1, F_1)\to \mathbb C, \,\,\,\, \text{Ext}^1(E_2, F_2)\to \mathbb C\, .$$ Since the reduced
virtual class is obtained by reducing only one trivial factor, a
trivial factor still remains. 
\vskip.1in

We now analyze the two surviving fixed components 
$\mathsf F^+$ and $\mathsf F^-$ 
associated to the splittings
 $$E_1\oplus E_2 = \mathcal O_X \ \oplus\ 
\mathcal I_W\otimes \mathcal O_X (-C) \ \ \ \ \text{and}\ \ \ \
E_1\oplus E_2 =  
\mathcal I_Z\otimes \mathcal O_X (-C)\ \oplus\ \mathcal O_X$$
respectively.
 The isomorphisms
$$\mathsf F^+\simeq \mathsf F^-\simeq X^{[n]}\times \mathbb P$$
are immediate. 

Let $\mathcal L$ be the hyperplane class on $\mathbb P,$ $$\mathcal L = 
\mathcal O_{\mathbb P} (1)\, .$$ The line bundle associated to the universal curve $\mathcal C\hookrightarrow X\times \mathbb P$ of the linear series $|H|$ is
 $$\mathcal O(\mathcal C)=H\otimes \mathcal L\, .$$ On  the universal surface of the fixed locus $$\mathsf F^{+} \times X\to \mathsf F^{+},$$ the universal subsheaves are $$\mathcal E_1=\mathcal O, \ \ \ \mathcal E_2=\mathcal I_{\mathcal W}\otimes \mathcal O(-\mathcal C)=\mathcal I_{\mathcal W}\otimes H^{-1}\otimes \mathcal L^{-1}\, .$$ 
To find the reduced virtual class of $\mathsf F^+$,
 we study the fixed part of the reduced
obstruction theory: \begin{eqnarray*}(\text{Tan}-\text{Obs})^{\text{fix}}&=&\text{Ext}^{\bullet}(\mathcal E_2, \mathcal F_2)+\mathbb C\\&=&\text{Ext}^{\bullet}(\mathcal I_{\mathcal W}\otimes H^{-1}\otimes \mathcal L^{-1}, \mathbb C-\mathcal I_{\mathcal W}\otimes H^{-1}\otimes \mathcal L^{-1})+\mathbb C\\ &=&\text{Ext}^{\bullet}(\mathcal I_{\mathcal W}\otimes H^{-1}\otimes \mathcal L^{-1}, \mathbb C) - \text{Ext}^{\bullet}(\mathcal I_{\mathcal W}, \mathcal I_{\mathcal W})+\mathbb C\\&=&\text{Ext}^{\bullet}(\mathcal I_{\mathcal W}\otimes H^{-1}, \mathbb C)\otimes \mathcal L- \text{Ext}^{\bullet}(\mathcal I_{\mathcal W}, \mathcal I_{\mathcal W})+\mathbb C\\&=&H^0(X, H)\otimes \mathcal L -\text{Ext}^{\bullet}(\mathcal O_{\mathcal W}\otimes H^{-1}, \mathbb C)\otimes \mathcal L -\text{Ext}^{\bullet}(\mathcal I_{\mathcal W}, \mathcal I_{\mathcal W})+\mathbb C\, .
\end{eqnarray*} The latter
 expression simplifies using the following four observations:
\begin{enumerate}
\item[(i)]
$\text{Ext}^1(\mathcal I_{W}, \mathcal I_{W})\simeq \text{Ext}^0(\mathcal I_W,\mathcal O_W)$ is the tangent space to $X^{[n]}$
at $W$,
\item[(ii)] $\text{Ext}^2(\mathcal I_{W}, \mathcal I_{W})=\text{Ext}^0(\mathcal I_W, \mathcal I_W)=\mathbb C$ by Serre duality,
\item[(iii)] 
$\text{Ext}^{0}(\mathcal O_{W}\otimes H^{-1}, \mathbb C)=
\text{Ext}^{1}(\mathcal O_{W}\otimes H^{-1}, \mathbb C)=0$
since $W$ is of dimension 0,
\item[(iv)]
 $\text{Ext}^{2}(\mathcal O_{W}\otimes H^{-1}, \mathbb C)=H^0(H^{-1}\otimes \mathcal O_W)^{\vee}$ by Serre duality.
\end{enumerate}
Applying (i-iv), we find
 \begin{eqnarray*}(\text{Tan}-\text{Obs})^{\text{fix}}&=&\mathbb C^{\ell+2}\otimes \mathcal L- \mathcal L\otimes \left((H^{-1})^{[n]}\right)^{\vee}+\text{Tan}_{X^{[n]}}-\mathbb C\\&=& \text{Tan\,}_{\mathbb P}- \mathcal L\otimes \left((H^{-1})^{[n]}\right)^{\vee}+\text{Tan}_{X^{[n]}}
\end{eqnarray*} where we have also 
used the Euler sequence $$0\to \mathcal O\to \mathbb C^{\ell+2} \otimes \mathcal L\to \text{Tan\,}_{\mathbb P}\to 0\, .$$ 
Therefore, over the fixed locus $\mathsf F^+$, we have the obstruction bundle $$\text{Obs}=\mathcal L\otimes \left((H^{-1})^{[n]}\right)^{\vee}.$$ As explained in Proposition 5.6 of \cite {BF}, the virtual class of the fixed locus $\mathsf{F}^+$ is obtained by taking the Euler class of the locally free \text{Obs}: $$\left[\mathsf F^{+}\right]^{\text{vir}}=\mathsf e\left(\mathcal L\otimes \left((H^{-1})^{[n]}\right)^{\vee}\right).$$ 
The analysis of the fixed part of the reduced obstruction theory for 
$\mathsf{F}^-$ is identical. We find that the virtual class of the
fixed locus is
 $$\left[\mathsf F^{-}\right]^{\text{vir}}=\mathsf e\left(\mathcal L\otimes \left((H^{-1})^{[n]}\right)^{\vee}\right)$$
just as for $\mathsf{F}^+$.

Turning to the normal bundle, we study
 the moving part of the reduced
obstruction theory on $\mathsf{F}^+$:
\begin{eqnarray*}
 \mathsf N^{+} &= &\text{Ext}^{\bullet}(E_1, F_2)+\text{Ext}^{\bullet}(E_2, F_1)\\
& = & \text{Ext}^{\bullet}(\mathcal O_X, F_2)\\
& =& H^{\bullet}(F_2)\\
&=& H^{\bullet}(\mathcal O_X)-H^{\bullet}(E_2)\\
&=&\mathbb C+\mathbb C-H^{\bullet}(E_2)\, .
\end{eqnarray*}
 In families, $H^{\bullet}(E_2)$
equals $$H^{\bullet}(H^{-1}\otimes \mathcal L^{-1}\otimes \mathcal I_{\mathcal W}) =\mathbb C^{\ell+2}\otimes \mathcal L^{-1}-H^{\bullet}(H^{-1}\otimes \mathcal O_{\mathcal W})\otimes \mathcal L^{-1}.$$
Hence, we find $$\mathsf N^{+}=\left(\mathbb C+\mathbb C+\mathcal L^{-1} \otimes (H^{-1})^{[n]}-\mathcal L^{-1} \otimes \mathbb C^{\ell+2} \right)[1]\, ,$$
where the $[1]$ indicates the torus weight.
Similarly,
 $$\mathsf N^{-}=\left(\mathbb C+\mathbb C+\mathcal L^{-1} \otimes (H^{-1})^{[n]}-\mathcal L^{-1} \otimes \mathbb C^{\ell+2} \right)[-1]\, ,$$

The normal complexes $\mathsf N^{\pm}$ have virtual rank $ n-\ell.$ 
Taking into account the opposite equivariant weights for the two fixed loci, we write the inverses of the Euler classes of the virtual normal bundles over the two fixed loci as
\begin{eqnarray*} \frac{1}{\mathsf e (\mathsf N^+)} &= &\frac{(1-\zeta)^{\ell+2}}{c_{+} ((H^{-1})^{[n]} \otimes \mathcal L^{-1} ) }\, , \\ \frac{1}{\mathsf e(\mathsf N^{-})}&= & (-1)^{n-\ell}\cdot \frac {(1+ \zeta)^{\ell+2}}{c_- ((H^{-1})^{[n]} \otimes \mathcal L^{-1})}.
\end{eqnarray*}
where $\zeta$ is the hyperplane class on $\mathbb{P}$.
We have used the Chern class notation
 $$c_+ = \sum_{i\geq 0} c_i\, \ \ \ \text{and} \ \ \
c_{-} = \sum_{i \geq 0} (-1)^i c_i\, .$$
\subsection{The calculation} \label{cat}

Let $\zeta$ also the denote the pull-back of the hyperplane class under the support morphism $$\pranktwo\to \mathbb P\, , \ \ \ \ \ \left[0\to E\to \mathbb C^2\otimes \mathcal O_X\to F\to 0\right]\mapsto \text{supp } F\, .$$ Consider the integrals $$ 0 =  I_k = \frac{(-1)^n}{2}\int_{[\pranktwo]^{\text{red}}} \, \zeta^{2k-1} \cdot 0^{2n+2 -2k}\, , \, \, \, \, \, \, 1 \leq k \leq n\, ,$$ against the reduced virtual class 
$$\left[\pranktwo\right ]^{\text{red}} \in A_{2n+1} \left (\left[\pranktwo\right]^{\text{red}} \right )\, .$$ The notation $0$ stands for the first Chern class of the trivial line bundle $$c_1(\mathcal O)=0\in A^1(\pranktwo).$$

We will express the vanishing integrals $I_k$ using the virtual localization formula \eqref{gpvi}. This requires specifying equivariant lifts of the classes involved. First, the reduced virtual class admits a canonical lift to the
$\mathbb C^{\star}$-equivariant cycle theory. The action of $\mathbb{C}^\star$ on $\zeta$ is trivial. Finally, we lift the 
$\mathbb{C}^\star$ action on $0$ as the
equivariant first Chern class of the representation $[1]$.

By examining the overall signs, the two non-vanishing fixed loci $\mathsf F^+$ and $\mathsf F^-$ can be seen to have identical contribution. We obtain:

\begin{eqnarray*} I_k &=& \int_{X^{[n]} \times \mathbb P} \zeta^{2k-1} \frac{(1-\zeta)^{\ell+2} \,\mathsf e \left((H^{-1})^{[n]} \otimes \mathcal L^{-1}\right)}{c \left((H^{-1})^{[n]} \otimes \mathcal L^{-1}\right )} \\ 
& = & \int_{X^{[n]} \times \mathbb P}\zeta^{2k-1} (1-\zeta)^{\ell+2} \left ( \sum_{i=0}^n c_{n-i} \left((H^{-1})^{[n]}\right) (-\zeta)^i \right ) \left ( \sum_{j\geq 0} \frac{s_j \left((H^{-1})^{[n]} \right)}{(1-\zeta)^{n+j}} \right ) \\
& = & \sum_{i, j \geq 0} \int_{X^{[n]} \times \mathbb P} (-1)^i\cdot \zeta^{2k + i -1} \, (1-\zeta)^{\ell+2 -n -j}\, \cdot c_{n-i} \left((H^{-1})^{[n]}\right) \cdot s_j \left((H^{-1})^{[n]}\right) \\
& = & (-1)^{\ell}\sum_{i=0}^{n} \binom{\ell+2 -i-2n}{\ell+2 - i -2k}  \, \int_{X^{[n]}} c_{n-i} \left((H^{-1})^{[n]}\right)\cdot s_{n+i} \left((H^{-1})^{[n]}\right)\, .
\end{eqnarray*}

\vskip.2in

\noindent Setting $$\alpha_i=\int_{X^{[n]}} c_{n-i}\left((H^{-1})^{[n]}\right)\cdot s_{n+i} \left((H^{-1})^{[n]}\right) \, \, \, \, \text{for} \, \, \, \, 0\leq i\leq n\, ,$$ we have proven the following result.

\begin{lemma}\label{identities} For all $1\leq k\leq n$, we have $$\sum_{i=0}^{n}\binom{\ell+2-i-2n}{\ell+2-i-2k}\alpha_i=0\, .$$
\end{lemma} 
\noindent In the statement above, we use the follow standard
conventions for the binomial:
\begin{enumerate}
\item[$\bullet$] $\binom{0}{0}=1$,
\item[$\bullet$]
$\binom{a}{b}=0$ for  $0\leq a<b$,
\item[$\bullet$]  $\binom{a}{b}=0$ for $ b< 0$, 
\item[$\bullet$] $\binom{-a}{b}=\frac{-a(-a-1)\cdots (-a-b+1)}{b!}$ for  $a>0,\, b\geq 0$. 
\end{enumerate}
Since $$c((H^{-1})^{[n]})\cdot s((H^{-1})^{[n]})=1\, ,$$ we obtain additionally $$\alpha_0+\alpha_1+\ldots+\alpha_n=0\, .$$

\section {Proof of Theorem 1} 
\label{pppp}
To prove Theorem \ref{main} for $H^{-1}$,
 we must show
$$\alpha_n={2^n}\binom{\ell-2n+2}{n}\, .$$ 
For the argument below, we will regard $\alpha_n$ as a degree $n$ polynomial in $\ell$, as shown in \cite {T}. The proof will be obtained by combining the following two statements:
 \begin {itemize}
\item [(i)] the polynomial $\alpha_n(\ell)$ has roots $\ell=2n-2,\, 2n-1,\, \ldots,\, 3n-3$,
\item [(ii)] the leading term of $\alpha_n(\ell)$ is $$\alpha_n(\ell)=\frac{2^n}{n!}\ell^{n}+ \ldots.$$
\end {itemize} 

\noindent {\it Proof of (i).} We will use Lemma \ref{identities} for the values 
\begin{equation}\label{cxxc}
\ell=2n-2, \ldots, 3n-3\, .
\end{equation} In fact, the Lemma yields trivial relations for $\ell>3n-3$. We fix a value of $\ell$ in the sequence \eqref{cxxc}, and show that for this value $$\alpha_n(\ell)=0.$$ All $\alpha$'s in the argument below will be evaluated at this fixed value of $\ell$, but for simplicity of notation, we will not indicate this explicitly.  
For convenience, write $$\ell=2n-2+p, \,\, 0\leq p\leq n-1\, ,$$ and relabel $$\beta_j=\alpha_{j+p},\,\, 1\leq j\leq n-p\, .$$ We will show that $$\beta_{n-p}=0\, .$$
Lemma \ref{identities} written for $k\mapsto n-k$ yields the relations $$\sum_{i=0}^{n} \binom{p-i}{p-i+2k}\alpha_i=0$$ for all $0\leq k\leq n-1.$ When $k=0$, the binomials corresponding to $i\leq p$ are equal to $1$, while those for $i>p$ are $0$. We obtain \begin{equation}\label{beta1}\alpha_0+\ldots+\alpha_p=0\implies \beta_1+\ldots+\beta_{n-p}=0\, .\end{equation}
When $1\leq k\leq n-1$, the binomials with $i\leq p$ are zero, so the only possible non-zero binomial numbers are obtained for $i>p.$ Writing $i=p+j$, the above identity becomes \begin{equation}\label{beta2}\sum_{j=1}^{n-p} \binom{-j}{2k-j}\beta_j=0\, .\end{equation} Note that $$\binom{-j}{2k-j}=(-1)^{j} \binom{2k-1}{j-1}\, .$$ We collect all the equations \eqref{beta1} and \eqref{beta2} for values $k\leq n-p-1$ derived above. Using the notation of the Lemma \ref{shhs}
below, the system of equations in $\beta_j's$ thus obtained can be written as $$A_{n-p} \begin{bmatrix} \beta_1\\\beta_2\\ \ldots \\ \beta_{n-p}\end{bmatrix}=0\, .$$By Lemma \ref{shhs}, $A_{n-p}$ is invertible. Hence, $$\beta_{n-p}=0$$ as claimed. 

\begin{lemma} \label{shhs} The $n\times n$ matrix $A_n$ with entries $$a_{ij}=(-1)^{j+1}\binom{2i-1}{j}\, , \ \ \ 0\leq i, j\leq n-1 $$  is invertible.  
\end{lemma} 
\proof 
Let $A$ be the $n\times n$ matrix with entries 
$$a_{ij}=\binom{m_i}{j}\, , \ \ \ 0\leq i, j\leq n-1\, . $$ 
If the $m_i$ are distinct, we will prove the  invertibility of $A$  via a
standard Vandermonde argument.  Since $A_n$  is obtained (up to the signs
$(-1)^{j+1}$ which only affect the sign of the determinant)
by the specialization
$$m_i=2i-1\, ,$$ the
Lemma follows.

If two $m$'s are equal, then the matrix $A$ has two equal rows, so 
$\det A$ vanishes. Therefore, $\det A$ is divisible by 
$$\prod_{i< j} (m_i-m_j)\, .$$ Since $A$ has degree $\frac{n(n-1)}{2}$ in the $m$'s, it suffices to show that the leading term is non-zero. If the leading term were zero, then $\det A=0$ for all choices of $m_i$. However, for $m_i=i$ the matrix $A$ is triangular with $1$'s on the diagonal.  
\qed

\vskip.1in
 
\noindent {\it Proof of (ii).} For the series $A_1(z)$ in \eqref{egleq}, we write $$A_1(z) = a_{11} z + a_{12} z^2 + a_{13} z^3 + \cdots\, .$$ By evaluating the coefficient of $z$ in \eqref {egleq}, we obtain
$$2\ell \,a_{11} = \int_X s_2 (H) = 2\ell \implies a_{11} =1\, .$$ 
The coefficient of $z^n$ in \eqref{egleq} yields 
\begin{eqnarray*}
N_{X,H,n} &=& \frac{(2\ell\,a_{11})^{n}}{n!} + \, \, \text{lower order terms in} \, \,\ell\\ & =& \frac{2^n}{n!}\, \ell^n + \ldots\, .
\end{eqnarray*} 

\vspace{8pt}
Statements (i) and (ii) imply Theorem 1 for primitively 
polarized $K3$ surfaces $(X,H)$ of degree $2\ell$ and
Picard rank $1$. By \eqref{egleq}, Theorem 1 then holds for
all pairs $(X,H)$ where $X$ is a nonsingular projective
$K3$ surface. \qed

\begin{remark} It is natural to consider the remaining integrals $$\int_{X^{[n]}} c_{n-i} (H^{[n]})\cdot s_{n+i}(H^{[n]}).$$ The case $i=n$ was the subject of Lehn's conjecture. On the other hand, $i=0$ corresponds to the curve calculations of \cite{Cott, LeB, Wang}:  
$$\int_{X^{[n]}} c_{n} (H^{[n]})\cdot s_{n}(H^{[n]})=\int_{C^{[n]}}s_n((H|_{C})^{[n]})=(-4)^n\binom{\frac{\ell}{2}}{n},$$ where $C$ is a smooth curve in the linear series $|H|$. The intermediate cases are unknown. However, when $i=n-1$, the methods of this work show that $$\int_{X^{[n]}} c_{1} (H^{[n]})\cdot s_{2n-1}(H^{[n]})=-2^{n}\binom{\ell-2n+2}{n-1} \cdot \left(\ell-\frac{3n-3}{2}\right).$$ Crucially, the proof of (i) above also establishes that $\alpha_{n-1}$, viewed as a degree $n$ polynomial in $\ell$, has $n-1$ roots at $$\ell=2n-2, 2n-3, \ldots, 3n-4.$$ Therefore, $\alpha_{n-1}$ is determined up to an ambiguity of two coefficients. These can be read off from the asymptotics $$\alpha_{n-1}=-\frac{1}{(n-1)!} \cdot (2\ell)^{n}+\frac{(5n-3)}{(n-2)!}\cdot (2\ell)^{n-1} +\ldots.$$ To prove this fact, it is more useful to consider the full Chern and Segre polynomials in the variables $x$ and $y$ $$c_x=1+c_1x+c_2x^2+\ldots,\,\,\, s_y=1+s_1y+s_2y^2+\ldots.$$ In the $K3$ case, we derive using \cite {EGL} the exponential form $$\sum_{n=0}^{\infty} z^{n}\int_{X^{[n]}} c_{x} (H^{[n]})\cdot s_{y}(H^{[n]})=\exp (2\ell\cdot A(z)+B(z))$$ for power series $$A(z)=a_{1}z+a_{2} z^2+\ldots, \,B(z)=b_1z+b_2z^2+\ldots$$ whose coefficients are polynomials in $x$ and $y$. Evaluating the coefficient of $z^n$, we obtain the asymptotics $$\int_{X^{[n]}} c_{x} (H^{[n]})\cdot s_{y}(H^{[n]})=\frac{a_1^n}{n!}\cdot (2\ell)^{n}+\left(\frac{a_1^{n-1}b_1}{(n-1)!}+\frac{a_1^{n-2}a_2}{(n-2)!}\right)\cdot (2\ell)^{n-1}+\ldots.$$ When $n=1$ this becomes $$\int_{X} c_x(H) \cdot s_y(H)=2y(y-x)\cdot \ell\implies a_1=y(y-x), \,\,\, b_1=0,$$ while for $n=2$, we find $$a_2=y^2(y-x)(2x-5y).$$ Therefore, 
$$\int_{X^{[n]}} c_{x} (H^{[n]})\cdot s_{y}(H^{[n]})=\frac{y^n(y-x)^n}{n!}\cdot(2\ell)^{n}+\frac{y^n (y-x)^{n-1}(2x-5y)}{(n-2)!} \cdot (2\ell)^{n-1}+\ldots.$$ Finally, we isolate the integral $\alpha_{n-1}$ by considering the coefficient of $xy^{2n-1}.$ 

\end{remark}
\section {Lehn's conjecture for $K$-trivial surfaces} \label{LLL}
We verify here several statements made in Section \ref{iiii} concerning
Lehn's conjecture:
\begin{itemize}
\item [(i)] Theorem \ref{main} is equivalent to Conjecture 1
for $K3$ surfaces.
\item [(ii)] Conjecture 1  for $K3$ sufaces
implies equations \eqref{a1} and \eqref{a4}.
\item [(iii)] Corollaries \ref{maincor} and \ref{maincor2} hold. 
\end{itemize}
{\it Proof of (i).} Let $(X,H)$ be a $K3$ surface
of degree $2\ell$.
We must prove $$\sum_{n=0}^{\infty} N_{X,H,n}\, z^n=\frac{(1-2w)^{2\ell+6}}{(1-6w+6w^2)^{\ell+2}}$$ 
for $N_{X,H,n}=2^{n}\binom{\ell-2n+2}{n}$ given by Theorem \ref{main}
and the change of variables
$$z=\frac{w(1-w)(1-2w)^4}{(1-6w+6w^2)^3}\, .$$
 The following Lemma applied to $$x=2z,\, \ \ y=2(w-w^2)$$ completes the proof of (i).

\begin{lemma}\label{ident} After the
change of variable 
$$x=\frac{y(1-2y)^2}{(1-3y)^3}\, ,$$ we have $$\sum_{n=0}^{\infty} x^n \binom{\ell-2n+2}{n} = \frac{(1-2y)^{\ell+3}}{(1-3y)^{\ell+2}}\, .$$
\end{lemma} 
\begin{proof} We carry out a further change of variables $$\frac{1-2y}{1-3y}=1+t\implies x=t(1+t)^2.$$ We must
 show 
$$\sum_{n=0}^{\infty} x^n \binom{\ell-2n+2}{n}=\frac{(1+t)^{\ell+3}}{1+3t}\, .$$ Denote the left side by $f_{\ell}(x)$. Pascal's identity shows the recursion \begin{equation}\label{recursion}f_{\ell+1}(x)=f_{\ell}(x)+xf_{\ell-2}(x).\end{equation}
The corresponding recursion for the right side  $\frac{(1+t)^{\ell+3}}{1+3t}$
is easily verified.
We prove the Lemma by induction on $\ell$. 

We first verify the base case consisting of three consecutive integers, for instance $$\ell=-4, -3, -2\, .$$ When $\ell=-4$, we must show  $$\sum_{n=0}^{\infty} x^n \binom{-2n-2}{n}=\frac{1}{(1+t)(1+3t)}=\left(\frac{dx}{dt}\right)^{-1}=\frac{dt}{dx}$$ which 
may be rewritten as \begin{equation}\label{teq}t=\sum_{n=0}^{\infty} \frac{x^{n+1}}{n+1} \binom{-2n-2}{n}\, ,\end{equation} where $t=g(x)$ is the solution of the equation $$t(1+t)^2=x\, ,$$ 
valid in a neighborhood of $x=0$ and with $g(0)=0$. 
The Taylor expansion \eqref{teq} 
for $g$ is equivalent to the value of the $n^{\text{th}}$-derivative $$\frac{d^n g}{dx^n}(0)=(n-1)!\binom{-2n}{n-1}\, .$$ The  $n^{\text{th}}$-derivative of the inverse function $g$ at $0$ is determined
 by the Schur-Jabotinsky Theorem \cite {J}: $g^{(n)}(0)$ equals $$(n-1)!\cdot \text{ coefficient of } t^{-1} \text { in the expansion of }x^{-n}.$$ Since $$x^{-n}=t^{-n} (1+t)^{-2n}\, ,$$ the above coefficient has value $\binom{-2n}{n-1}$ as
required. 

Using the identity $$\frac{1+t}{1+3t}=(1+t)-\frac{3}{(1+t)(1+3t)}\cdot t(1+t)^2=(1+t)-3f_{-4}(x)\cdot x$$ and equation \eqref{teq}, we obtain 
\begin{eqnarray*}
\frac{1+t}{1+3t}&=&1+ \sum_{n=0}^{\infty} \frac{x^{n+1}}{n+1} 
\binom{-2n-2}{n}-3\sum_{n=0}^{\infty} x^{n+1} \binom{-2n-2}{n}\\
&=&1+\sum_{n=0}^\infty x^{n+1} \binom{-2n-2}{n+1}\\
&=&f_{-2}(x)\, ,
\end{eqnarray*}
 which verifies the Lemma for $\ell=-2$. 

For $n\geq 1$, we have the identity
 $$\binom{-2n-1}{n}=\frac{3}{2}\binom{-2n}{n}\, .$$ 
Taking into account corrections coming from $n=0$, we obtain 
$$\sum_{n=0}^{\infty} x^{n} \binom{-2n-1}{n}=\frac{3}{2} \sum_{n=0}^{\infty} x^{n} \binom{-2n}{n} -\frac{1}{2}\, .$$ Equivalently, $$f_{-3}(x)\, =\, \frac{3}{2} f_{-2}(x)-\frac{1}{2}\, =\, \frac{3}{2} \cdot \frac{1+t}{1+3t}-\frac{1}{2}\, =\, \frac{1}{1+3t}\, ,$$
which verifies the remaining base case.
\end{proof}

\vspace{8pt}
\noindent {\it Proof of (ii).} We combine equations \eqref{egleq} and \eqref{series} for
a $K3$ surface $(X,H)$ of degree $2\ell$
 to obtain $$\exp 
\left(2\ell A_1(z) +24 A_4(z)\right)=
\frac{(1-2w)^{2\ell+6}}{(1-6w+6w^2)^{\ell+2}}\, .$$ 
Setting $\ell=0$ yields $$\exp (24 A_4(z))=\frac{(1-2w)^6}{(1-6w+6w^2)^2}$$ and thus $$\exp(2\ell A_1(z))=\frac{(1-2w)^{2\ell}}{(1-6w+6w^2)^{\ell}}\implies A_1(z)=\frac{1}{2} \log \left(\frac{1-4w+4w^2}{1-6w+6w^2}\right).$$ Setting as before $$z=\frac{1}{2}t(1+t)^2\implies \frac{1-4w+4w^2}{1-6w+6w^2}=1+t,$$ we obtain equation \eqref{a1}. Equation \eqref{a4} follows by the same method. \vskip.1in

\vspace{8pt}
\noindent {\it Proof of (iii).} Let $(A,H)$ be an
abelian or bielliptic surface of degree $2\ell$.
Equation \eqref{egleq} and the proof of (ii) yield
 $$\sum_{n=0}^{\infty} N_{A,H,n}\, z^n=\exp (2 \ell A_1(z))=\frac{(1-2w)^{2\ell}}{(1-6w+6w^2)^{\ell}}\, .$$ Corollary \ref{maincor} is equivalent to the identity $$\sum_{n=0}^{\infty} z^n \frac{2^n\ell}{n} \binom{\ell-2n-1}{n-1}=\frac{(1-2w)^{2\ell}}{(1-6w+6w^2)^{\ell}}.$$ 
After the change of variables $2z=x$ and $2(w-w^2)=y$, we obtain
 \begin{equation}\label{abelian}\sum_{n=0}^{\infty} x^{n} \frac{\ell}{n} \binom{\ell-2n-1}{n-1}=\left(\frac{1-2y}{1-3y}\right)^{\ell}\end{equation} in the notation of Lemma \ref{ident}. 
We have $$\frac{dy}{dx}=\left(\frac{dx}{dy}\right)^{-1}=\frac{(1-3y)^4}{1-2y}\implies d\left(\frac{1-2y}{1-3y}\right)=\frac{(1-3y)^2}{1-2y}dx\, .$$ Writing Lemma \ref{ident} for $\ell\mapsto \ell-5$, we obtain $$\sum_{n=0}^{\infty} x^{n} \binom{\ell-2n-3}{n}=\frac{(1-2y)^{\ell-2}}{(1-3y)^{\ell-3}}\, $$ 
and therefore also $$\sum_{n=0}^{\infty} x^{n} \binom{\ell-2n-3}{n} dx=\frac{(1-2y)^{\ell-1}}{(1-3y)^{\ell-1}} d\left(\frac{1-2y}{1-3y}\right).$$ 
After integration, we obtain
 $$\sum_{n=0}^{\infty} \frac{x^{n+1}}{n+1} \binom{\ell-2n-3}{n} = \frac{1}{\ell} \frac{(1-2y)^{\ell}}{(1-3y)^{\ell}}-\frac{1}{\ell}$$ which yields \eqref{abelian} after the shift $n\to n-1$. This proves Corollary \ref{maincor}. 

\vspace{8pt}
Let $(E,H)$ be an
Enriques surface of degree $2\ell$.
Equation \eqref{egleq} yields
 $$\sum_{n=0}^{\infty} N_{E,H,n}\, z^n=\exp(2\ell A_1(z)+12A_4(z))=\frac{(1-2w)^{2\ell+3}}{(1-6w+6w^2)^{\ell+1}}\, .$$ 
Using the change of variables  $2z=x$ and $2(w-w^2)=y$
as above, we obtain $$\sum_{n=0}^{\infty} \frac{N_{E,H,n}}{2^n} x^n=\frac{(1-2y)^{\ell+\frac{3}{2}}}{(1-3y)^{\ell+1}}\, .$$ In particular, by Lemma \ref{ident},
 \begin{equation}\label{enriques}\left(\sum_{n=0}^{\infty} 
\frac{N_{E,H,n}}{2^n} x^n\right)^2=\sum_{n=0} x^n \binom{2\ell-2n+2}{n}=\sum_{n=0}^{\infty} \frac{N_{X,H,n}}{2^n} x^n,\end{equation} where the $(X,H)$
is the $K3$ double covering $E$ with the divisor class
determined by pull-back.
The proof of Corollary \ref{maincor2} is complete.

\begin{corollary} If $Y$ is a minimal elliptic surface and $H=mf$ is a multiple of the fiber class, then
$$\sum_{n} \frac{N_{Y,H,n}}{2^n} x^n=\left(\sum_{n=0}^{\infty} x^n \binom{-2n+2}{n}\right)^{\chi(\mathcal O_Y)/2}\, .$$
\end{corollary}

\begin{proof}
Equation \eqref{egleq} yields
$$\sum_{n} N_{Y,H,n}\, z^n=\exp(12\chi(\mathcal O_Y) \cdot A_4(z))=\left(\frac{(1-2w)^{3}}{1-6w+6w^2}\right)^{\chi(\mathcal O_Y)}\, .$$ 
The Corollary follows
via the change of variables $2z=x$ and $2(w-w^2)=y$
and Lemma \ref{ident}.
 \end{proof}

\section{Moduli of $K3$ surfaces} \label{MMM}

\subsection{Tautological classes}
Let $\mathcal{M}_{2\ell}$ be the moduli space of quasi-polarized
$K3$ surfaces $(X,H)$ of degree $2\ell$. The Noether-Lefschetz loci
define
classes in the Chow ring $\AA(\mathcal{M}_{2\ell},\mathbb{Q})$. Let
$$\NL(\mathcal{M}_{2\ell}) \subset \AA(\mathcal{M}_{2\ell},\mathbb{Q})$$
be the subalgebra generated by Noether-Lefschetz loci (of all codimensions).
A basic result conjectured in  \cite{MP} and proven in \cite{Ber} is the
isomorphism
$$\mathsf{NL}^1(\mathcal{M}_{2\ell}) = \mathsf{A}^1(\mathcal{M}_{2\ell},\mathbb{Q})\, .$$

Another method of constructing classes in $\AA(\mathcal{M}_{2\ell},\mathbb{Q})$
is the following.
Let
$$ \pi: \mathcal{X} \rightarrow \mathcal{M}_{2\ell}$$
be the universal surface. Let 
$$\mathcal{H} \rightarrow \mathcal{X} \ \ \ \text{and} \ \ \
\mathcal{T}_\pi\rightarrow \mathcal{X} 
$$
the universal quasi-polarization (canonical up to a twist
by an element of $\text{Pic}(\mathcal{M}_{2\ell})$)
and the relative tangent bundle.
The $\kappa$ classes are defined by
$$\kappa_{a,b}  = \pi_*\left(c_1(\mathcal{H})^a \cdot 
c_2(\mathcal{T}_\pi)^b\right) 
\ \in \mathsf{A}^{a+2b-2}(\mathcal{M}_{2\ell},\mathbb{Q})\ .$$

There is no need to include a $\kappa$ index for the
first Chern class of $\mathcal{T}_\pi$ since
$$c_1(\mathcal{T}_\pi) = -\pi^*\lambda$$ where $\lambda=c_1(\mathbb E)$ is the first Chern class of the Hodge line bundle
$$\mathbb{E} \rightarrow \mathcal{M}_{2\ell}$$
with fiber $H^0(X,K_X)$ over the moduli
point $(X,H)\in \mathcal{M}_{2\ell}$. The Hodge class $\lambda$ is known to be supported on Noether-Lefschetz divisors.{\footnote{In \cite {MP},
$\lambda$ is considered a degenerate Noether-Lefschetz divisor. See
\cite[Theorem 3.1]{DM} to express $\lambda$ in terms of 
proper Noether-Lefschetz divisors.}} 
The ring generated by $\lambda$ has been determined in \cite{kvg}. 

On the Noether-Lefschetz locus $\mathcal{M}_{\Lambda}\subset \mathcal{M}_{2\ell}$
 corresponding to 
the Picard lattice $\Lambda$, richer $\kappa$
classes may be defined
 using {\it all} the powers of {\it all} the universal line bundles
associated to $\Lambda$. If $H_1,\ldots,H_r$ is a
basis of $\Lambda$ and 
$$\mathcal{H}_1,\ldots,\mathcal{H}_r \rightarrow \mathcal{X}_\Lambda$$
are the associated universal bundles 
(canonical up to a twist
by an element of $\text{Pic}(\mathcal{M}_{\Lambda})$)
over the universal
surface
$$\mathcal{X}_\Lambda \rightarrow \mathcal{M}_{\Lambda}\, ,$$
we define the $\kappa^\Lambda$ classes by
$$\kappa_{a_1,\ldots,a_r,b}  = 
\pi_*\left(c_1(\mathcal{H}_1)^{a_1}\cdots
c_1(\mathcal{H}_r)^{a_r}
 \cdot 
c_2(\mathcal{T}_\pi)^b\right) 
\ \in \mathsf{A}^{\sum_i a_i+2b-2}(\mathcal{M}_{\Lambda})\ .$$

We define the {\it tautological ring} of the moduli space
of $K3$ surfaces,
$$\RR(\mathcal{M}_{2\ell}) \subset \AA(\mathcal{M}_{2\ell},\mathbb{Q})\, ,$$
to be the subring generated by the push-forwards from
the Noether-Lefschetz loci of 
all monomials in
the $\kappa$ classes. By definition, 
$$\NL(\mathcal{M}_{2\ell}) \subset \RR(\mathcal{M}_{2\ell})\, .$$

\subsection{Relations}
Consider the universal surface
$$ \pi:\mathcal{X} \rightarrow \mathcal{M}_{2\ell},$$
and the $\pi$-relative Quot scheme $\prankpi$ parametrizing 
quotients
 $$\mathbb C^r \otimes \mathcal O_X \to F\to 0$$  
where $F$ is a rank 0 coherent sheaf on $(X,H)$
satisfying 
$$c_1(F)=dH\ \ \ \text{and} \ \ \ \chi(F)=\chi\, .$$

Just as in Section \ref{alice},
the $\pi$-relative
Quot scheme admits a reduced virtual class
via the $\pi$-relative obstruction theory. The push-forwards 
\begin{equation}\label{rrr}
\pi_*\left( \gamma\cdot 0^{k} \, \cap \,
{\left[\prankpi\right]^{\text{red}}}
\right) \, \in\,  \AA(\mathcal{M}_{2\ell},\mathbb{Q}) 
\end{equation}
 vanish for all $k> 0$ and for all choices of Chow classes $\gamma$.
There is a natural torus action on $\prankpi$
via the torus action on $\mathbb{C}^r$. 
Virtual localization \cite{GP} applied to  the push-forwards \eqref{rrr}
yields relations in $\RR(\mathcal{M}_{2\ell})$. The
Noether-Lefschetz loci appear naturally when the curve class
$dH$ splits into non-multiples of $H$.

There is no difficulty in writing the resulting relations
in $\RR(\mathcal{M}_{2\ell})$ in terms of push-forwards from
the $\pi$-relative Hilbert schemes of points.
To show the relations are {\it non-trivial} requires
non-vanishing results in the intersection theory
of the $\pi$-relative Hilbert scheme. 
A thorougher study of the relations in $\RR(\mathcal{M}_{2\ell})$ 
obtained by localizing the virtual class of $\pi$-relative Quot schemes
will appear elsewhere.  

\subsection{Virtual localization}\label{s4} We illustrate the program proposed above with several examples. For simplicity, we will only consider the case $d=1$. The relative Quot scheme $$p:\pranktwopi\to \mathcal M_{2\ell}$$ has relative reduced virtual dimension equal to $2n+1$
 where 
$$n=\chi+\ell\, . $$
As before, there is a  support morphism $$\pranktwopi \to \mathbb P(\pi_{\star} \mathcal H), \ \ \ \left[0\to E\to \mathbb C^2\otimes \mathcal O_X\to F\to 0\right]\mapsto \text{supp } F.$$ Consider the codimension $1$ vanishing push-forwards $$\frac{(-1)^n}{2}\cdot p_{\star} \left(0^{2n+2-2k}\cdot \zeta^{2k}\cap \left[\pranktwopi\right]^{\text{red}} \right)\in \mathsf A^{1}(\mathcal M_{2\ell}, \mathbb Q),$$ for $0\leq k\leq n.$ We  evaluate these expressions via equivariant localization
for the torus action induced from the splitting $$\mathbb C^2=\mathbb C[0]+ \mathbb C[1].$$ 
 
Over the generic point $(X, H)$ of the moduli space, the fixed loci over the Quot scheme correspond to kernels of the form $$E=I_Z + I_W\otimes \mathcal O_X(-C)$$ where $$\ell(Z)=z,\,\, \ell(W)=w,\,\, z+w=n$$ and $C$ is a curve in the linear series $|H|$. The weights are distributed in two possible ways over the summands. 

There are additional fixed loci over surfaces in the Noether-Lefschetz divisors corresponding to a nontrivial-splitting $[H]=[C_1]+[C_2]$ so that $$E=I_Z\otimes \mathcal O_X(-C_1)+I_W\otimes \mathcal O_X(-C_2).$$  
The curve classes $C_1, C_2$ are effective, so they intersect $H$ nonegatively.  Since $$C_1\cdot H+C_2\cdot H=2\ell, \,\,\,\ell(Z)+\ell(W)=n-C_1\cdot C_2,$$ there are finitely many such Noether-Lefschetz divisors and fixed loci over them. 
We will address them later.

For now, to describe the {\it generic} fixed loci and their contributions, we introduce the following notation: \begin{itemize}
\item [--] Over the moduli space $\mathcal M_{2\ell}$, the bundle $$\mathbb V=\pi_{\star} \mathcal H$$ has rank $\ell+2$. We write 
$$\mathbb P=\mathbb P(\mathbb V)\ \  \text{ and } \ \ \mathcal L=\mathcal O_{\mathbb P}(1)$$ for the associated relative linear series $|\mathcal H|$ over the moduli space and the corresponding tautological line bundle  respectively.

\vspace{8pt}
\item [--] Over the relative Hilbert scheme $\mathcal X^{[z]}\to \mathcal M_{2\ell}$, we introduce the rank $z$ bundle $$\mathcal H^{[z]}=\text{pr}_{2\star} (\text{pr}_1^{\star}{\mathcal H}\otimes \mathcal O_{\mathcal Z}),$$ obtained as the push forward from the universal surface of the twisted universal subscheme $$\mathcal Z\subset \mathcal X\times_{\mathcal M} 
\mathcal X^{[z]}\to \mathcal X^{[z]}\, .$$ 
\end{itemize}
The fixed loci are $$\mathbf A[z, w]= \mathcal X^{[z]}\times_{\mathcal M} \left(\mathcal X^{[w]}\times_{\mathcal M}\mathbb P\right).$$  Over the universal surface $$\mathbf A[z, w]\times_{\mathcal M}\mathcal X\, ,$$
the universal subsheaf equals $$\mathcal E=I_{\mathcal Z}\,\oplus \,I_{\mathcal W}\otimes \mathcal L^{-1}\otimes \mathcal H^{-1}.$$ The non-reduced obstruction bundle over $\mathbf A[z, w]$ is found by considering the fixed part of the virtual tangent bundle: $$\text{Ext}^{\bullet}(\mathcal E, \mathcal F)^{\text{fix}}=\text{Ext}^{\bullet}(I_{\mathcal Z}, \mathcal O_{\mathcal Z})+\text{Ext}^{\bullet}(I_{\mathcal W}\otimes \mathcal L^{-1}\otimes \mathcal H^{-1}, \mathcal O_X-I_{\mathcal W}\otimes \mathcal L^{-1}\otimes \mathcal H^{-1}).$$ To calculate this bundle explicitly, we follow the method of Section \ref{bob} for a {\it fixed} surface. The only difference is that in the use of relative duality, the relative canonical bundle will yield a copy of the Hodge class $$K_{\mathcal X/\mathcal M}=\pi^{\star} \mathbb E\, .$$ 
We only record here the final answer (for $z\neq 0$ and $w\neq 0$): $$\text{Obs}^{\text{red}}=\mathbb E^{\vee}+\mathbb E^{\vee}\otimes \left(\mathcal O^{[z]}\right)^{\vee}+\mathbb E^{\vee}\otimes \mathcal L \otimes \left((\mathcal H^{-1})^{[w]}\right)^{\vee}\, .$$ 
Just as in the case of a fixed surface where almost all fixed loci have a trivial summand,
the fixed loci here keep an additional copy of the Hodge bundle. The corresponding contribution $$\mathsf{Cont}\, \mathbf A[z, w]=m(z, w) \cdot \lambda$$ is a multiple of the Hodge class. In our examples below, $m(z, w)$ can be found via a calculation over the Hilbert scheme of a fixed surface. (A separate calculation shows that the same conclusion also holds for $w=0$.)

There are however two exceptions, where the formulas above must be modified. These correspond to the two fixed loci $\mathsf F^+$ and $\mathsf F^-$ where the splitting of the kernel is 
$$E=\mathcal O_X[0]\,\oplus\, I_W\otimes \mathcal O_X(-C)[1] \ \ \ \text{or} \ \ \
E=I_Z\otimes \mathcal O_X(-C)[0]\,\oplus \,\mathcal O_X[1]$$ for a curve $C$ in the linear series $|H|$. We have $$\mathsf F^+\simeq\mathsf F^-\simeq \mathcal X^{[n]}\times_{\mathcal M} \mathbb P\, .$$ The following hold true:

\begin{itemize}
\item [--] The obstruction bundle equals $$\text{Obs}^{\text{red}}=\mathbb E^{\vee}\otimes\, \mathcal L\otimes \left((\mathcal H^{-1})^{[n]}\right)^{\vee}\, ,$$ so that $$\left[\mathsf F^+\right]^{\text{vir}}= \left[\mathsf F^-\right]^{\text{vir}}=(-1)^n\cdot \mathsf e\left(\mathbb E\otimes\, \mathcal L^{-1}\otimes (\mathcal H^{-1})^{[n]}\right)\, .$$
\item [--] The virtual normal bundle is $$\mathsf N^{\pm}=\left(\mathbb C+\mathbb E^{\vee}+\mathcal L^{-1} \otimes (\mathcal H^{-1})^{[n]}-\mathcal L^{-1} \otimes \mathbb V^{\vee} \otimes \mathbb E^{\vee}\right)[\pm 1]\, .$$ The formulas of Section \ref{bob} are special cases. 
\end{itemize}
Therefore, 
\begin{eqnarray*}
\frac{1}{\mathsf e(\mathsf N^{+})} &=& \frac{c_-(\mathcal L \otimes \mathbb V\otimes\mathbb E)}{1-\lambda}\cdot \frac{1}{c_+(\mathcal L^{-1}\otimes (\mathcal H^{-1})^{[n]})}\, ,\\
\frac{1}{\mathsf e(\mathsf N^{-})} &=&(-1)^{n-\ell}\cdot \frac{c_+(\mathcal L \otimes \mathbb V\otimes\mathbb E)}{1+\lambda}\cdot \frac{1}{c_-(\mathcal L^{-1}\otimes (\mathcal H^{-1})^{[n]})}\, .
\end{eqnarray*} The localization contributions of the fixed loci $\mathsf F^+$ and $\mathsf F^-$ are equal and  take the form $$q_{\star}\left(\zeta^{2k} \cdot \mathsf e\left(\mathbb E\otimes\, \mathcal L^{-1}\otimes (\mathcal H^{-1})^{[n]}\right)\cdot \frac{c_-(\mathcal L \otimes \mathbb V\otimes\mathbb E)}{1-\lambda}\cdot \frac{1}{c_+(\mathcal L^{-1}\otimes (\mathcal H^{-1})^{[n]})}\right),$$ where $$q:\mathcal X^{[n]}\times_{\mathcal M} \mathbb P\to \mathcal M_{2\ell}$$ is the projection. 
This expression can be expanded as in Section \ref{cat}.

The following push-forwards from the relative Hilbert scheme
are central 
to the calculation:
 $$\gamma_i=\text{pr}_{\star}\left(s_{n+i+1}\left((\mathcal H^{-1})^{[n]}\right)\cdot c_{n-i}\left((\mathcal H^{-1})^{[n]}\right)\right)\in \mathsf A^1(\mathcal M_{2\ell}),\,\,\,\, 0\leq i\leq n .$$ For the sum of contributions of 
fixed loci dominating $\mathcal{M}_{2\ell}$, we obtain
$$(-1)^{\ell+1}\cdot \left(\sum_{i=0}^{n}\binom{\ell+1-2n-i}{\ell+1-2k-i}\cdot \gamma_i\, +\, 
a_{k}\cdot c_1(\mathbb V)\,  +\, b_k\cdot \lambda\right),$$ for constants $a_k$ and $b_k$ that depend on $\ell$, $k$, $n$. For instance, $$a_{k}=\sum_{i=0}^{n} \binom{\ell+1-2n-i}{\ell+2-2k-i} \alpha_i\, .$$ The expression thus obtained must be in the span of the Noether-Lefschetz divisors. Similar relations can written down in higher codimension. 

The class $$\gamma_n=\text{pr}_{\star}\left(s_{2n+1}\left((\mathcal H^{-1})^{[n]}\right)\right)$$ is the codimension one analogue of the Segre integrals which are the subject of Lehn's conjecture.  
While there are procedures to calculate $\gamma_n$, we do not yet 
have a closed form expression. Nonetheless, by the recursions of \cite {EGL}, it can be seen that $$\gamma_n\in  \mathsf R^{\star}(\mathcal M_{2\ell})$$ is a combination of the classes $\kappa_{3, 0}$, $\kappa_{1, 1}$, $\lambda$ with coefficients which are polynomials in $\ell$. 

\subsection {Examples} We present here explicit relations
involving $\kappa$ classes and
Noether-Lefschetz divisors. As explained in the introduction of \cite {MP}, there are two essentially equivalent ways of thinking of the Noether-Lefschetz
 divisors. Our convention here is that we specify the {\it Picard class}. In other words, we consider surfaces for which there exists a class $\beta$ with given lattice $$\begin{bmatrix} H^2 & \beta\cdot H\\ \beta\cdot H & \beta^2\end{bmatrix}.$$

\subsubsection{Degree two} The simplest case is of 
$K3$ surfaces of degree $H^2=2$. Relevant to the discussion are the following Noether-Lefschetz divisors:

\begin{itemize} 
\item [(i)] The divisor $\mathcal P$ corresponding to the lattice $$\begin{bmatrix} 2 & 1 \\ 1 & 0 \end{bmatrix}.$$ 
\item [(ii)] The reduced divisor $\mathcal S$ corresponding to the lattice $$\begin{bmatrix} 2 & 0 \\ 0 & -2 \end{bmatrix}.$$  Here, the linear series $|H|$ fails to be ample along the $(-2)$-curve. The divisor $\mathcal S$ contains $\mathcal P$ as a component.
\end{itemize} 
By work of O'Grady \cite {OG}, the following isomorphism holds $$\mathsf A^1(\mathcal M_2, \mathbb Q)= \mathbb Q\cdot \left[\mathcal P\right]+\mathbb Q\cdot \left[\mathcal S\right].$$ However, there are more tautological classes to consider. In addition to the Hodge class $\lambda$, we also have 
$$\kappa_{3, 0}=\pi_{\star} (c_1(\mathcal H)^3)\ \ \ \text{and} \ \ \  \kappa_{1, 1}=\pi_{\star} (c_1(\mathcal H)\cdot c_2(\mathcal T_{\pi}))\, .$$ While the classes $\kappa_{3, 0}$ and $\kappa_{1, 1}$ are canonically defined only up to a choice of a quasi-polarization, the difference $$\kappa_{3, 0}-4\kappa_{1, 1}$$ is independent of choices. There are two relations connecting the four tautological classes $$\kappa_{3, 0}-4\kappa_{1, 1},\,\, \lambda,\,\,\, \left[\mathcal P\right],\,\,\, \left[\mathcal S\right].$$ \begin {proposition}\label{p1} We have
 $$\kappa_{1, 1}-4\kappa_{3, 0}-18\lambda +12\left[\mathcal P\right]=0\, ,$$
 $$\kappa_{1, 1}-4\kappa_{3, 0}+\frac{9}{2}\lambda-\frac{24}{5}\left[\mathcal P\right]-\frac{3}{10}\left[\mathcal S\right]=0.$$
\end {proposition}

\begin{proof}
The relations are obtained by considering the following numerics:
\begin{itemize}
\item [(i)] For the first relation, we use the relative Quot scheme of short exact sequences $$0\to E\to \mathbb C^2\otimes \mathcal O_X\to F\to 0$$ where $$c_1(F)=H\, ,\ \ \chi(F)=0.$$ The relative virtual dimension equals $3$, and we evaluate the pushforward $$p_{\star}(0^4\cap [\mathcal Q^{\,\pi}_{\,H,\, 0}(\mathbb C^2)]^{\text{red}})\in \mathsf A^1(\mathcal M_2)\, .$$
\item [(ii)] For the second relation, we use the numerics $$c_1(F)=H\, ,\ \  \chi(F)=1.$$ The relative virtual dimension equals $5$, and we consider the pushforward $$p_{\star}(0^6\cap [\mathcal Q^{\,\pi}_{\, H, \,1}(\mathbb C^2)]^{\text{red}})\in \mathsf A^1(\mathcal M_2)\, .$$  
\end{itemize}

We provide details only for case (i)  since (ii) is parallel. 
\begin {itemize}
\item [--] Over a generic surface $(X, H)$, the total contribution can be found by the method of Section \ref{s4}. The answer is $$2\gamma_1-4c_1(\mathbb V)=2\kappa_{3, 0}-4c_1(\mathbb V)=\frac{4\kappa_{3, 0}-\kappa_{1, 1}}{3}+6\lambda\, .$$ 
Grothendieck-Riemann-Roch was used to express $$c_1(\mathbb V)=-\frac{3\lambda}{2}+\frac{\kappa_{1, 1}}{12}+\frac{\kappa_{3, 0}}{6}.$$

\item [--] Over the locus $\mathcal P$, the classes $H_1=\beta$, $H_2=H-\beta$ are effective with $$H_1\cdot H_2=1\, , \ \ H_1^2=0\, ,\ \  H_2^2=0\, .$$ The torus fixed kernels are $$E=\mathcal O(-C_1)+\mathcal O(-C_2)$$ for curves $C_1$ and $C_2$ in the linear series $|H_1|$ and $|H_2$. There are two fixed loci, with identical contributions, obtained by switching the weights. The multiplicity of the divisor $\mathcal P$ in the final calculation equals $-2$ over each fixed locus. 

The $-2$ can be confirmed by working over a fixed surface. Each fixed locus is isomorphic to $$|H_1|\times |H_2|=\mathbb P^1\times \mathbb P^1.$$ The fixed part of the tangent obstruction theory equals $$\text{Ext}^{\bullet}(\mathcal O(-C_1), \mathcal O_{C_1})+\text{Ext}^{\bullet}(\mathcal O(-C_2), \mathcal O_{C_2}).$$ The Hodge factor of the reduced theory cancels the additional Hodge class coming from the normal bundle of the divisor $\mathcal P$ in the moduli stack, see \cite {OG}. The virtual class of the fixed locus agrees with the usual fundamental class. 

Similarly, the normal bundle is found by considering the moving part $$\mathsf N=\text{Ext}^{\bullet}(E, F)^{\text{mov}}=\mathsf N_1[1]+\mathsf N_2[-1]$$ where $$\mathsf N_1=\text{Ext}^{\bullet}(\mathcal O(-C_1), \mathcal O_{C_2})=H^{\bullet}(\mathcal O(C_1)|_{C_2})=H^{\bullet}(\mathcal O(C_1))-H^{\bullet}(\mathcal O(C_1-C_2))$$ $$\mathsf N_2=\text{Ext}^\bullet(\mathcal O(-C_2), \mathcal O_{C_1})=H^{\bullet}(\mathcal O(C_2)|_{C_1})=H^{\bullet}(\mathcal O(C_2))-H^{\bullet}(\mathcal O(C_2-C_1)).$$ In families, these formulas also include the the tautological bundles $\mathcal L_1$ and $\mathcal L_2$ over the linear series. We have 
$$\mathsf N_1=H^{\bullet}(H_1)\otimes \mathcal L_1-H^{\bullet}(H_1\otimes H_2^{-1})\otimes \mathcal L_1\otimes \mathcal L_2^{-1}=\mathcal L_1\otimes \mathbb C^2-\mathcal L_1\otimes \mathcal L_{2}^{-1}$$ and analogously $$\mathsf N_2=\mathcal L_2\otimes \mathbb C^2-\mathcal L_2\otimes \mathcal L_1^{-1}.$$ The final multiplicity becomes $$\int_{\mathbb P^1\times \mathbb P^1} \frac{1}{\mathsf e(\mathsf N)}=\int_{\mathbb P^1\times \mathbb P^1}\frac{(1+\zeta_1-\zeta_2)\cdot (-1-\zeta_1+\zeta_2)}{(1+\zeta_1)^2(-1+\zeta_2)^2}=-2\, .$$ 
\end {itemize}
\end{proof}

\begin{remark} We can write down relations in higher codimension as well. For instance, evaluating the pushforward
$$p_{\star}(0^4\cdot \zeta\cap [\mathcal Q^{\,\pi}_{\,H,\, 0}(\mathbb C^2)]^{\text{red}})\in \mathsf A^2(\mathcal M_2)$$
yields  the identity
$$11\kappa_{4, 0} - \kappa_{2, 1} -
3 \kappa_{3, 0}^2- 4 \kappa_{3, 0} \cdot \left[\mathcal P\right]  + 6 \kappa_{3, 0}\cdot \lambda - 8 \lambda\cdot \left[\mathcal P\right] +4\lambda^2=0.$$
\end{remark}

\subsubsection {Degree four} Our next example concerns the case of quartic 
$K3$ surfaces with $H^2=4$. We single out three Noether-Lefschetz divisors: 

\begin {itemize} 
\item [(i)] The reduced divisor $\mathcal P_1$ corresponding to the lattice $$\begin{bmatrix} 4 & 1\\ 1 &0\end{bmatrix}.$$
\item [(ii)] The reduced divisor $\mathcal P_2$ underlying the Noether-Lefschetz locus corresponding to the lattice $$\begin{bmatrix} 4 & 2\\ 2 &0\end{bmatrix}.$$ The divisor $\mathcal P_2$ is not irreducible, since it receives contributions from hyperlliptic locus and from $\mathcal P_1$.
\item [(iii)] The reduced divisor $\mathcal S$ corresponding to the lattice $$\begin{bmatrix} 4 & 0\\ 0 &-2\end{bmatrix}.$$ Here, the linear series $|H|$ fails to be ample along the $(-2)$-curve.
\end {itemize} 
 A $K3$ surface in the complement of the union of the 
 three divisors $\mathcal P_1,\, \mathcal P_2,\, \mathcal S$ is necessarily a 
nonsingular quartic in $\mathbb P^3$. Again by work of \cite {OG}, we have an isomorphism $$\mathsf A^1(\mathcal M_4, \mathbb Q)= \mathbb Q\cdot \left[\mathcal P_1\right]+\mathbb Q\cdot \left[\mathcal P_2\right]+\mathbb Q\cdot \left[\mathcal S\right].$$ The class $$\kappa_{1, 1}-2\kappa_{3, 0}$$ is independent of the choice of polarization. The following relations are obtained:
\begin {proposition} \label{p2} We have
$$10(\kappa_{1,1}-2\kappa_{3, 0})-38\lambda+8[\mathcal P_1]-2[\mathcal P_2]-3[\mathcal S]=0\, ,$$
$$4(\kappa_{1,1}-2\kappa_{3, 0})-26\lambda+16[\mathcal P_1]+2[\mathcal P_2]-[\mathcal S]=0\, .$$
\end {proposition}

The proof of Proposition \ref{p2} uses the relative Quot scheme of short exact sequences $$0\to E\to \mathbb C^2\otimes \mathcal O_X\to F\to 0$$ where the numerics are chosen so that $$c_1(F)=H\, ,\ \ \  \chi(F)=0.$$ The relative virtual dimension equals $5$. 
The first relation  is obtained by evaluating the pushforward $$p_{\star}(0^6\cap [\mathcal Q^{\,\pi}_{\,H, \,0}(\mathbb C^2)]^{\text{red}})\in \mathsf A^1(\mathcal M_4)\, ,$$ while second relation  is obtained by evaluating $$p_{\star}(0^4\cdot \zeta^2 \cap [\mathcal Q^{\,\pi}_{\,H, \,0}(\mathbb C^2)]^{\text{red}})\in \mathsf A^1(\mathcal M_4)\, .$$ The details are similar to the proof of Proposition \ref{p1}. 
 
\subsubsection{Degree 6} Consider the reduced Noether-Lefschetz divisors 
$$\mathcal P_1, \,\mathcal P_2,\, \mathcal P_3, \,\mathcal S$$
corresponding to the lattices 
$$\begin{bmatrix} 6 & 1\\ 1 &0\end{bmatrix}, \,\begin{bmatrix} 6 & 2\\ 2 &0\end{bmatrix}, \,\begin{bmatrix} 6 & 3\\ 3 &0\end{bmatrix}, \,\begin{bmatrix} 6 & 0 \\ 0 & -2 \end{bmatrix}.$$ 
The divisors $\mathcal P_2$ and $\mathcal P_3$ are not irreducible, since they contain $\mathcal P_1$ as a component. It is shown in \cite {OG} that the Picard rank is $4$. We expect two relations between the invariant combination of the $\kappa$'s, the Hodge class and the $4$ boundary divisors:
\begin{proposition} We have
$$(-3\kappa_{1,1}+4\kappa_{3, 0})+21\lambda-16[\mathcal P_1]-[\mathcal P_2]+[\mathcal S]=0,$$
$$-11(3\kappa_{1, 1}-4\kappa_{3, 0})+280\lambda-230 \left[\mathcal P_1\right]-38\left[\mathcal P_2\right]-2\left[\mathcal P_3\right]+10\left[\mathcal S\right]=0.$$\end{proposition}
\noindent These relations can be found by integrating over $\mathcal Q^{\,\pi}_{\,H, \,-1}(\mathbb C^2)$ and over $\mathcal Q^{\,\pi}_{\,H, \,0}(\mathbb C^2).$

\subsubsection {Degree 8} In this case, the Picard rank is $4$, see \cite {OG}. We expect two relations between the invariant combination of the $\kappa$'s, the Hodge class and the $4$ boundary divisors $$\mathcal P_1, \,\mathcal P_2,\,\mathcal P_3, \,\,\mathcal S$$
corresponding to the lattices 
$$\begin{bmatrix} 8 & 1\\ 1 &0\end{bmatrix}, \,\begin{bmatrix} 8 & 2\\ 2 &0\end{bmatrix},\, \,\begin{bmatrix} 8 & 3\\ 3 &0\end{bmatrix},\,\,\begin{bmatrix} 8 & 0 \\ 0 & -2 \end{bmatrix}.$$ Integrating over $\mathcal Q^{\,\pi}_{\,H, \,-2}(\mathbb C^2)$ and $\mathcal Q^{\,\pi}_{\,H, \,-1}(\mathbb C^2)$ we find:
\begin{proposition}
$$\frac{8}{3}(\kappa_{1,1}-\kappa_{3, 0})-24\lambda+24[\mathcal P_1]+4[\mathcal P_2]-[\mathcal S]=0,$$ 
$$-\frac{128}{3}(\kappa_{1, 1}-\kappa_{3, 0})+424\lambda-440 \left[\mathcal P_1\right]-92\left[\mathcal P_2\right]-8\left[\mathcal P_3\right]+15\left[\mathcal S\right]=0.$$
\end{proposition}
\subsection{Conjecture}
The main conjecture suggested by the abundance of relations
obtained by localizing the virtual class of $\pi$-relative Quot schemes
is the following.
\begin{conjecture} \label{fff} For all $\ell\geq 1$,  we have
$\NL(\mathcal{M}_{2\ell}) = \RR(\mathcal{M}_{2\ell})$.
\end{conjecture} 
If true, Conjecture \ref{fff} would lead to a much simpler
picture of the additive structure of $\RR(\mathcal{M}_{2\ell})$
since  good approaches to the span of the Noether-Lefschetz
classes are available \cite{B,KM}. We further speculate that the relations
obtained by localizing
 the virtual class of $\pi$-relative Quot schemes are sufficient
to prove Conjecture \ref{fff}. We have checked this in small degree and small codimension.

\end{document}